\documentclass[12pt,a4paper]{article}%
\usepackage{amssymb}
\usepackage{amsmath}
\usepackage{amsfonts}
\usepackage{graphicx}%
\providecommand{\U}[1]{\protect\rule{.1in}{.1in}}
\newtheorem{theorem}{Theorem}

\newtheorem{corollary}[theorem]{Corollary}

\newtheorem{lemma}[theorem]{Lemma}

\newtheorem{proposition}[theorem]{Proposition}
\newtheorem{remark}[theorem]{Remark}

\newenvironment{proof}[1][Proof]{\noindent\textbf{#1.} }{\ \rule{0.5em}{0.5em}}
\begin{document}

\title{The Dirichlet problem for the minimal surface equation on unbounded helicoidal
domains of $\mathbb{R}^{m}$}
\author{Ari Aiolfi, Caroline Assmann, Jaime Ripoll }
\maketitle

\begin{abstract}
We consider a helicoidal group $G$ in $\mathbb{R}^{n+1}$ and unbounded
$G$-invariant $C^{2,\alpha}$-domains $\Omega\subset\mathbb{R}^{n+1}$ whose
helicoidal projections are exterior domains in $\mathbb{R}^{n}$, $n\geq2$. We
show that for all $s\in\mathbb{R}$, there exists a $G$-invariant solution
$u_{s}\in C^{2,\alpha}\left(  \overline{\Omega}\right)  $ of the Dirichlet
problem for the minimal surface equation with zero boundary data which
satisfies $\sup_{\partial\Omega}\left\vert \operatorname{grad}u_{s}\right\vert
=\left\vert s\right\vert $. Additionally, we provide further information on
the behavior of these solutions at infinity.

\end{abstract}

\section{Introduction}

The Dirichlet problem for the minimal surface equation (mse) in $\mathbb{R}%
^{m}$, $m\geq2$, namely,%
\begin{equation}
\left\{
\begin{array}
[c]{c}%
\mathcal{M}\left(  u\right)  :=\operatorname{div}\left(  \tfrac
{\operatorname{grad}u}{\sqrt{1+\left\vert \operatorname{grad}u\right\vert
^{2}}}\right)  =0\text{ in }\Omega,~u\in C^{2}\left(  \Omega\right)  \cap
C^{0}\left(  \overline{\Omega}\right) \\
u|_{\partial\Omega}=\varphi
\end{array}
\right.  \label{PD}%
\end{equation}
where $\Omega\subset\mathbb{R}^{m}$ is a $C^{2}$-domain and $\varphi\in
C^{0}\left(  \partial\Omega\right)  $ is given a priori, has been intensively
explored in the last decades. One of the most general answers to the Dirichlet
problem (\ref{PD}) for bounded domain was given by H. Jenkins and J. Serrin in
\cite{JS}. They showed that (\ref{PD}) is solvable for arbitrary $\varphi\in
C^{0}\left(  \partial\Omega\right)  $ if only if $\Omega$ is mean convex.
Moreover, they noted that if $\varphi\in C^{2}\left(  \partial\Omega\right)
$, a bound on the oscillation of $\varphi$ in terms of the second order norm
of $\varphi$ should be enough to ensure the solvability of (\ref{PD}) on
arbitrary bounded domains (Theorem 2 of \cite{JS}).

The study of the Dirichlet problem for the mse on unbounded domains began with
J. C. C. Nitsche in the so called \emph{exterior domains }that is, when
$\mathbb{R}^{m}-\Omega$ is relatively compact (Section 4 of \cite{N}). Since
then several authors continue the investigation of the Dirichlet problem for
the mse in exterior domains (\cite{O}, \cite{Kr}, \cite{KT}, \cite{CK}%
,\cite{K} \cite{R}, \cite{RT} \cite{ABR}).

The Dirichlet problem (\ref{PD}) for more general unbounded domains reduces,
to authors knowledge, to few works: when $m=2$, Rosenberg and Sa Earp
(\cite{ERo}) proved that (\ref{PD}) has a solution if $\Omega\subset
\mathbb{R}^{2}$ is convex subset distinct of a half plane for any continuous
boundary data $\varphi$. In the half plane case, in the case that $\varphi$ is
bounded, Collin and Krust (\cite{CK}) proved that if $\Omega$ is a convex set
distinct of a half plane, then the solution is unique and, if $\Omega$ is a
half plane then there is a unique solution with linear growth. For an
arbitrary dimension $m$, Z. Jin, in \cite{J}, proved that (\ref{PD}) has a
solution if $\Omega$ is a mean convex domain contained in some special like
parabola-shape region or in the complement of a cone in $\mathbb{R}^{m}%
$.\ More recently N. Edelen and Z. Wang proved that if $\Omega\varsubsetneq
\mathbb{R}^{n}$ is an open convex domain (e.g. a half-space) and $\varphi\in
C^{0}\left(  \partial\Omega\right)  $ is a linear function, then any solution
of (\ref{PD}) must also be linear.

In our work we obtain an extension of the exterior Dirichlet problem for the
minimal surface equation in $\mathbb{R}^{m}$, $m\geq3,$ in the following
sense:\ we say that a domain is $k-$bounded, $0\leq k\leq m$, if it is bounded
in $k$ directions of the space $\mathbb{R}^{m}$ (as a direction we mean an
equivalence class of parallel lines of $\mathbb{R}^{m}$). As so, a domain of
$\mathbb{R}^{m}$ is relatively compact if and only if it is $m-$bounded. In
our main results we study the Dirichlet problem for the mse on certain domains
$\Omega$ of $\mathbb{R}^{m}$, $m\geq3,$ which complement $\mathbb{R}%
^{m}-\Omega$ is $(m-1)-$bounded, and for zero boundary data. We recall that
Theorem 3.5 of \cite{CK} proves the existence of solutions of the Dirichlet
problem for the mse on a strip, a special $1-$bounded domain of $\mathbb{R}%
^{2}$, for arbitrary continuous bounded boundary data.

To state precisely our theorems we need to recall a result of the third author
and F. Tomi (Theorem 2 of \cite{RT}) which asserts that if $\Omega$
is\ $G$-invariant $C^{2,\alpha}$-domain for $m\geq3$, where $G$ is a subgroup
of $ISO\left(  \mathbb{R}^{m}\right)  $ that acts freely and properly on
$\mathbb{R}^{m}$, such that $P\left(  \Omega\right)  $ is a bounded and mean
convex domain, then (\ref{PD}) has an unique $G$-invariant solution for any
$G-$invariant boundary data $\varphi\in C^{2,\alpha}\left(  \partial
\Omega\right)  $, where $P:\mathbb{R}^{m}\longrightarrow\mathbb{R}^{m}/G$ is
the projection through the orbits of $G$ and $\mathbb{R}^{m}/G$ is endowed
with a metric such that $P$ becomes a Riemannian submersion.

Related to the above result, we would like to mention that the use of a Lie
group of symmetries to study minimal surfaces was first considered by Wu-ye
Hsiang and Blaine Lawson in \cite{HL}. Although proving distinct facts, we can
say that Proposition 3 of \cite{RT} has the same spirit of Theorem 2 of
\cite{HL}. Also related to these results, there is the idea of using Lie
groups of symmetries to the study of minimal graphs (constant mean curvature
graphs and more general PDE's too), as Killing graphs in warped products. This
technique was first considered by Marcos Dajczer and the third author of this
work in \cite{DR} \ and, since then, many works have been done extending and
generalizing the results of \cite{DR}, as \cite{DHL}, \cite{DL} and
\cite{CHHL}.

Let $\lambda\in\mathbb{R}$, $a\in\mathbb{R-\{}0\}$ and $i,j,k\in\left\{
1,...,n+1\right\}  $ be given with any two $i,j$ and $k$ distinct. Consider
the helicoidal\ like group $G\equiv G_{\lambda,a}^{i,j,k}$ in $\mathbb{R}%
^{n+1}$, $n\geq2$, determined by the one parameter subgroup of isometries
$G=\left\{  \varphi_{t}\right\}  _{t\in\mathbb{R}}$, where $\varphi
_{t}:\mathbb{R}^{n+1}\longrightarrow\mathbb{R}^{n+1}$ is given by%
\begin{equation}
\varphi_{t}\left(  p\right)  =\alpha\left(  p\right)  e_{i}+\beta\left(
p\right)  e_{j}+\gamma\left(  p\right)  e_{k}+%
{\textstyle\sum\limits_{l\neq i,j,k}}
x_{l}e_{l}, \label{G}%
\end{equation}
where $p=\left(  x_{1},...,x_{n+1}\right)  $, $\left\{  e_{i}\right\}
_{i=1}^{n+1}$ is usual orthonormal basis of $\mathbb{R}^{n+1}$,
\[
\alpha\left(  p\right)  =x_{i}\cos\lambda t+x_{j}\sin\lambda t\text{, }%
\beta\left(  p\right)  =x_{j}\cos\lambda t-x_{i}\sin\lambda t
\]
and $\gamma\left(  p\right)  =x_{k}+at$.

Let $\pi:\mathbb{R}^{n+1}\longrightarrow\left\{  x_{k}=0\right\}
\equiv\mathbb{R}^{n}$ be the \emph{helicoidal} projection determined by $G$,
that is,%
\begin{equation}
\pi\left(  p\right)  =\widehat{\alpha}\left(  p\right)  e_{i}+\widehat{\beta
}\left(  p\right)  e_{j}+%
{\textstyle\sum\limits_{l\neq i,j,k}}
x_{l}e_{l}, \label{hp}%
\end{equation}
where
\[
\widehat{\alpha}\left(  p\right)  =x_{i}\cos\frac{\lambda x_{k}}{a}-x_{j}%
\sin\frac{\lambda x_{k}}{a}\text{, }\widehat{\beta}\left(  p\right)
=x_{j}\cos\frac{\lambda x_{k}}{a}+x_{i}\sin\frac{\lambda x_{k}}{a}.
\]

Set
\begin{equation}
M:=(\mathbb{R}^{n},\left\langle ,\right\rangle _{G}), \label{M}%
\end{equation}
where $\left\langle \text{ },\text{ }\right\rangle _{G}$ is the metric such
that $\pi$ becomes a Riemannian submersion (clearly $G$ acts freely and
properly in $\mathbb{R}^{n+1}$ and $\left\{  x_{k}=0\right\}  \equiv
\mathbb{R}^{n}$ is a slice relatively to the orbits $Gp=\left\{  \varphi
_{t}\left(  p\right)  \text{, }t\in\mathbb{R}\right\}  $). One may see that
the map $\psi:\mathbb{R}^{n+1}/G\rightarrow\mathbb{R}^{n}$ given by $\psi
=\pi\circ P^{-1}$ is well defined and is an isometry with the metrics
mentioned above. From the isometric submersion theory, it follows that $M$ is complete.

Let $\Omega\subset\mathbb{R}^{n+1}$ be a $G$-invariant domain of class
$C^{2,\alpha}$ and set $\Lambda=\pi\left(  \Omega\right)  $. Let $d_{E}\left(
p\right)  =d_{E}\left(  p,\partial\Omega\right)  $, $p\in\Omega$, be the
(Euclidean) distance in $\mathbb{R}^{n+1}$ to $\partial\Omega$ and let
$d\left(  q\right)  =d_{M}\left(  q,\partial\Lambda\right)  $, $q\in\Lambda$,
be the distance in $M$ to $\partial\Lambda$. Given $u\in C^{2}\left(
\Omega\right)  \cap C^{0}\left(  \overline{\Omega}\right)  $ and $\varphi\in
C^{0}\left(  \partial\Omega\right)  ,$ $G$-invariant functions, that is,
$u=v\circ\pi$ for some $v\in C^{2}\left(  \Lambda\right)  \cap C^{0}\left(
\overline{\Lambda}\right)  $ and $\varphi=\psi\circ\pi$ for some $\psi\in
C^{0}\left(  \partial\Lambda\right)  $, it follows from Proposition 3 of
\cite{RT} that $u$ is solution of (\ref{PD}) (relatively to $m=n+1$) if, and
only if,%

\begin{equation}
\left\{
\begin{array}
[c]{c}%
\mathfrak{M}\left(  v\right)  :=\operatorname{div}_{M}\left(  \tfrac{\nabla
v}{\sqrt{1+\left\vert \nabla v\right\vert ^{2}}}\right)  -\tfrac{1}%
{\sqrt{1+\left\vert \nabla v\right\vert ^{2}}}\left\langle \nabla
v,J\right\rangle _{M}=0\text{ on }\Lambda\\
v|_{\partial\Lambda}=\psi
\end{array}
\right.  \label{PDM}%
\end{equation}
where $\nabla$ and $\operatorname{div}_{M}$ are the gradient and divergence in
$M$, respectively, and
\begin{equation}
J\left(  \pi\left(  p\right)  \right)  =d\pi_{p}\left(  \overrightarrow
{H}_{Gp}\left(  p\right)  \right)  ,\text{ }p\in\mathbb{R}^{n+1}, \label{J}%
\end{equation}
where $\overrightarrow{H}_{Gp}$ is the mean curvature vector of the
$1$-dimensional submanifold $Gp$ of $\mathbb{R}^{n+1}$. Moreover, $\left\vert
\overline{\nabla}u\right\vert =\left\vert \nabla v\right\vert \circ\pi$, where
$\overline{\nabla}$ denotes the gradient in $\mathbb{R}^{n+1}$.
\begin{figure}[h]
\centering
\fbox{\includegraphics[scale=0.7]{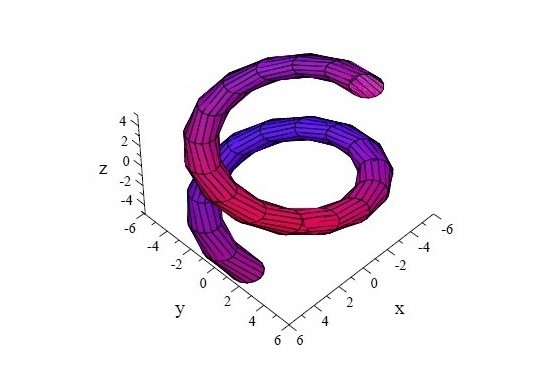}}\caption{A $G$-invariant
domain in $\mathbb{R}^{3}$}%
\end{figure}\newline

(\footnote{The white region in Figure 1 is a $G$-invariant domain
$\Omega\subset\mathbb{R}^{3}$ (with $\lambda=a=1$), whose boundary is the
surface $\Psi:\left[  0,\pi\right]  \times\mathbb{R\longrightarrow R}^{3}$
given by%
\[
\Psi\left(  t,z\right)  =\left(  \cos t\cos z+\left(  5+\sin t\right)  \sin
z,\left(  5+\sin t\right)  \cos z-\cos t\sin z,z\right)  .
\]
\par
Note that $\mathbb{R}^{3}-\overline{\Omega}$ is an example of a $2$-bounded
domain in $\mathbb{R}^{3}$.}) When $\Lambda$ is a bounded and mean convex
domain is proved in \cite{RT}, as mentioned above, that there is an unique
$G$-invariant solution $u\in C^{2,\alpha}\left(  \overline{\Omega}\right)  $
of (\ref{PD}) for all $G$-invariant $\varphi\in C^{2,\alpha}\left(
\partial\Omega\right)  $. In this note we work with the case where
$\Lambda=\pi\left(  \Omega\right)  $ is an exterior domain in $\mathbb{R}^{n}$
and the boundary data is zero. We observe that $\Lambda$ is an exterior domain
in $M$ if and only if $\Omega$ is $n-$bounded.

Regarding the condition of zero boundary data, we recall an old but quite
suggestive result to our work of Osserman that proves that in $\mathbb{R}^{2}$
with the Euclidean metric, there is a boundary data on a disk $D$ for which
there is no solution to the exterior Dirichlet problem in $\mathbb{R}^{2}-D$.
This strongly suggests, since $K_{M}>0$, that the zero boundary data condition
can not also be dropped out in our case. But we don't have a counter example.

In order to state our main results, we remember that, relatively to an
exterior domain $\mathbb{R}^{n}-\overline{\mathfrak{B}}_{\rho}\left(
p_{0}\right)  $ in $\mathbb{R}^{n}$, $n\geq2$, where $\mathfrak{B}_{\rho
}\left(  p_{0}\right)  $ is an open ball of $\mathbb{R}^{n}$ centered at
$p_{0}\in\mathbb{R}^{n}$ and of radius $\rho>0$, the function
\begin{equation}
v_{\rho}\left(  p\right)  :=\rho\int_{1}^{\frac{\tau}{\rho}}\frac{dt}%
{\sqrt{t^{2\left(  n-1\right)  }-1}}\text{, }\tau=\left\vert p-p_{0}%
\right\vert \text{, }p\in\mathbb{R}^{n}-\mathfrak{B}_{\rho}\left(
p_{0}\right)  , \label{ncat}%
\end{equation}
is the solution relatively to the Dirichlet problem (\ref{PD}) which
satisfies
\[
\underset{p\rightarrow\partial\mathfrak{B}_{\rho}\left(  p_{0}\right)  }{\lim
}\left\vert \overline{\nabla}v_{\rho}\left(  p\right)  \right\vert
=\infty\text{, }\underset{\left\vert p\right\vert \rightarrow\infty}{\lim
}\left\vert \overline{\nabla}v_{\rho}\left(  p\right)  \right\vert =0\text{ }%
\]
(a half-catenoid). If $n\geq3$, $v_{\rho}$ is bounded and its height at
infinity, which we denote by $h\left(  n,\rho\right)  $, is
\begin{equation}
h\left(  n,\rho\right)  =\rho h\left(  n,1\right)  =\rho\int_{1}^{\infty}%
\frac{dt}{\sqrt{t^{2\left(  n-1\right)  }-1}}. \label{hro}%
\end{equation}

In all of the results from now on, $G\equiv G_{\lambda,a}^{i,j,k}$ is as
defined in (\ref{G}), with $\lambda,a,i,j,k$ fixed.

We prove:

\begin{theorem}
\label{MT}Let $\Omega$ $\subset\mathbb{R}^{n+1}$, $n\geq2$, be a $G$-invariant
and $C^{2,\alpha}$-domain such that $\pi\left(  \Omega\right)  $ is an
exterior domain of $\mathbb{R}^{n}=\left\{  x_{k}=0\right\}  $, where $G$ and
$\pi$ are as defined in (\ref{G}) and (\ref{hp}), respectively. Let
$\varrho>0$ be the radius of smallest geodesic ball of $\mathbb{R}^{n}$ which
contain $\partial\pi\left(  \Omega\right)  $, centered at origin of
$\mathbb{R}^{n}$ if $\lambda\neq0$. Then, given $s\in\mathbb{R}$, there is a
$G$-invariant solution $u_{s}\in C^{2,\alpha}\left(  \overline{\Omega}\right)
$ of the Dirichlet problem (\ref{PD}) with $u_{s}|_{\partial\Omega}=0$, such
that:\newline i) $\sup_{\overline{\Omega}}\left\vert \overline{\nabla}%
u_{s}\right\vert =\sup_{\partial\Omega}\left\vert \overline{\nabla}%
u_{s}\right\vert =\left\vert s\right\vert $;\newline ii) $u_{s}$ is unbounded
if $n=2$ and $s\neq0$ and, if $n\geq3$, either
\[
\sup\left\vert u_{s}\right\vert \leq h\left(  n,\varrho\right)
\]
or there is a complete, non-compact, properly embedded $n$-dimensional
submanifold $N\subset\Omega$, such that
\[
\underset{d_{E}\left(  p\right)  \rightarrow+\infty}{\lim}u_{s}\left\vert
_{N}\right.  \left(  p\right)  =h\left(  n,\varrho\right)  ,
\]
where $h\left(  n,\varrho\right)  $ is given by (\ref{hro});\newline iii)
$u_{s}$ satisfies%
\[
\underset{d_{E}\left(  p\right)  \rightarrow\infty}{\lim}\left\vert
\overline{\nabla}u_{s}\left(  p\right)  \right\vert =0
\]
if $\lambda=0$ or $s=0$ or, if $\lambda\neq0$, $3\leq n\leq6$ and $u_{s}$is bounded.
\end{theorem}

Additional informations relatively to set of $G$-invariant solutions of the
Dirichlet problem (\ref{PD}) also are obtained under the assumption that
$M-\nolinebreak\pi\left(  \Omega\right)  $ satisfies the interior sphere
condition of radius $r>0$, that is, for each $q\in\nolinebreak\partial
\pi\left(  \Omega\right)  $, there is a geodesic sphere $S_{q}$ of $M$ of
radius $r$ contained in $M-\pi\left(  \Omega\right)  $ such that $S_{q}$ is
tangent to $\partial\pi\left(  \Omega\right)  $ at $q$ and $r$ is maximal with
this property.

Given $a\in\mathbb{R}- \left\{  0\right\}  $, $\lambda\in\mathbb{R} $,
$n\geq3$ and $r>0$ set
\begin{equation}
C=C\left(  r,n,\lambda,a\right)  :=\frac{2\left\vert a\right\vert \left(
n-1\right)  +\left\vert \lambda\right\vert r}{2\left\vert a\right\vert r}.
\label{C}%
\end{equation}
Let $\varsigma>C$ be the solution of the equation%
\begin{equation}
\cosh\left(  \frac{\mu}{\sqrt{\mu^{2}-C^{2}}}\right)  =\frac{\mu}{C}\text{,
}\mu>C, \label{eqmi}%
\end{equation}
and set%
\begin{equation}%
\mathcal{L}%
=%
\mathcal{L}%
\left(  r,n,\lambda,a\right)  :=\left\{
\begin{array}
[c]{c}%
\frac{1}{\sqrt{\varsigma^{2}-C^{2}}}\text{ if }\lambda\neq0\\
h\left(  n,r\right)  \text{ if }\lambda=0
\end{array}
\right.  , \label{L}%
\end{equation}
where $h\left(  n,r\right)  $ is given by (\ref{hro}).

\begin{theorem}
\label{CMT}Let $\Omega$ $\subset\mathbb{R}^{n+1}$, $n\geq3$, be a
$G$-invariant and $C^{2,\alpha}$-domain such that $\pi\left(  \Omega\right)  $
is an exterior domain of $\mathbb{R}^{n}=\left\{  x_{k}=0\right\}  $, where
$G$ and $\pi$ are as defined in (\ref{G}) and (\ref{hp}), respectively. Assume
that $M-\pi\left(  \Omega\right)  $ satisfies the interior sphere condition of
radius $r>0$, where $M$ is the $n$-dimensional Riemannian manifold given by
(\ref{M}). Let $%
\mathcal{L}%
=%
\mathcal{L}%
\left(  r,n,\lambda,a\right)  $ be as defined in (\ref{L}), where $\lambda$
and $a$ are given in the definition of $G$. Then, given $c\in\lbrack0,%
\mathcal{L}%
]$, there is a $G$-invariant solution $u_{c}\in C^{2}\left(  \Omega\right)
\cap C^{0}\left(  \overline{\Omega}\right)  $ of (\ref{PD}) with
$u_{c}|_{\partial\Omega}=0$, such that
\[
\underset{d_{E}\left(  p\right)  \rightarrow\infty}{\lim u_{c}\left(
p\right)  }=c.
\]
In particular, if $c\in\lbrack0,%
\mathcal{L}%
)$ then $u_{c}\in C^{2,\alpha}\left(  \overline{\Omega}\right)  $.
\end{theorem}

We note that our approach is not applicable for general boundary data.
Moreover, we were not able to prove that the solutions $u_{s}$ obtained in
The\nolinebreak orem \ref{MT} have a limit at infinity and this could be the
subject of a future research.

\section{Preliminaries}

We first observe that, relatively to the PDE given in (\ref{PDM}), the maximum
and comparison principles work (see Section 3 of \cite{RT}).

In this section, we give further informations on $M$ and we provide the basic
results to construct barriers relative to the Dirichlet problem (\ref{PDM})
when the boundary data is zero. We shall use the meaning of the indexes $i$
and $j$ as defined in (\ref{G}).

\begin{lemma}
\label{CGp}Let $G$ be as defined in (\ref{G}). Given $p=\left(  x_{1}%
,...,x_{n+1}\right)  \in\mathbb{R}^{n+1}$, the orbit $Gp$ has constant
curvature
\begin{equation}
H_{Gp}=\frac{\lambda^{2}\sqrt{x_{i}^{2}+x_{j}^{2}}}{\lambda^{2}(x_{i}%
^{2}+x_{j}^{2})+a^{2}}. \label{HG}%
\end{equation}
In particular,
\begin{equation}
\underset{p\in\mathbb{R}^{n+1}}{\sup}H_{Gp}=\frac{\left\vert \lambda
\right\vert }{2\left\vert a\right\vert }, \label{sHG}%
\end{equation}
and the supremum is attended, if $\lambda\neq0$, at those orbits through the
points $p$ in $\mathbb{R}^{n+1}$ such that $\left\vert \left(  x_{i}%
,x_{j}\right)  \right\vert =\left\vert a\right\vert /\left\vert \lambda
\right\vert $.
\end{lemma}

\begin{proof}
An arch length parametrization of $Gp$ is given by%
\[
\gamma_{p}(s)=F\left(  p\right)  e_{i}+G\left(  p\right)  e_{j}+H\left(
p\right)  e_{k}+\sum\limits_{l\neq i,j,k}x_{l}e_{l}.
\]
where
\[
F\left(  p\right)  =x_{i}\cos(A\left(  p\right)  s)+x_{j}\sin(A\left(
p\right)  s)\text{, }G\left(  p\right)  =x_{j}\cos(A\left(  p\right)
s)-x_{i}\sin(A\left(  p\right)  s)
\]
and%
\[
H\left(  p\right)  =x_{k}+\frac{aA\left(  p\right)  s}{\lambda},
\]
with%
\begin{equation}
A(p)=\frac{\lambda}{\sqrt{\lambda^{2}(x_{i}^{2}+x_{j}^{2})+a^{2}}}\text{.}
\label{A(p)}%
\end{equation}
The mean curvature vector of $Gp$ is then%
\begin{equation}
\overrightarrow{H}_{Gp}(\gamma_{p}(s))=\gamma_{p}^{\prime\prime}%
(s)=-A^{2}\left(  p\right)  \left[  F\left(  p\right)  e_{i}+G\left(
p\right)  e_{j}\right]  , \label{vcm}%
\end{equation}
and, consequently, the mean curvature of $Gp$ in $\mathbb{R}^{n+1}$ is given
by (\ref{HG}). Since the mean curvature of the orbit $Gp$ only depends on the
Euclidean distance of $x_{i}e_{i}+x_{j}e_{j}$ to the origin, setting
$\sigma\left(  p\right)  =\left\vert \left(  x_{i},x_{j}\right)  \right\vert
$, we have that $\xi:[0,\infty)\longrightarrow\nolinebreak\mathbb{R}$ given
by
\begin{equation}
\xi\left(  \sigma\right)  =\frac{\lambda^{2}\sigma}{\lambda^{2}\sigma
^{2}+a^{2}}, \label{csi}%
\end{equation}
has a maximal absolute point at $\sigma_{0}=\left\vert a\right\vert
/\left\vert \lambda\right\vert $ if $\lambda\neq0$ and the result follows.
\end{proof}

Since $\pi:\mathbb{R}^{n+1}\longrightarrow M$ is a Riemannian submersion,
given two orthogonal vector fields $X,Y\in\chi\left(  M\right)  $ and their
respective horizontal lift $\overline{X},\overline{Y}$, we know that%
\[
K_{M}\left(  X,Y\right)  =K_{\mathbb{R}^{n+1}}\left(  \overline{X}%
,\overline{Y}\right)  +\frac{3}{4}\left\vert \left[  \overline{X},\overline
{Y}\right]  ^{v}\right\vert _{\mathbb{R}^{n+1}}^{2}%
\]
where $K$ and $\left[  \overline{X},\overline{Y}\right]  ^{v}$ means,
respectively, the sectional curvature and the vertical component of $\left[
\overline{X},\overline{Y}\right]  $, that is, the component which is tangent
to the orbits $Gp$, $p\in\mathbb{R}^{n+1}$. As $K_{\mathbb{R}^{n+1}}\left(
\overline{X},\overline{Y}\right)  =0$, it follows that $K_{M}\left(
X,Y\right)  \geq0$ and, therefore, $Ric_{M}\geq0$ (straightforward, but quite
extensive calculations, give us that, in fact, $K_{M}>0$ with $K_{M}%
\rightarrow0$ at infinity).

\begin{lemma}
\label{nv}Let $\Lambda$ be an exterior domain in $M$. Denote by $\nu$ the
horizontal lift of $\nabla d$, where $d=d_{M}\left(  .,\partial\Lambda\right)
$. Then
\begin{equation}
\left\langle \nabla d,J\right\rangle _{M}\circ\pi=\left\langle \nu
,\overrightarrow{H}_{G}\right\rangle \label{hl}%
\end{equation}
where $\left\langle ,\right\rangle $ is the Euclidean metric and $J$ is given
by (\ref{J}). In particular
\[
-H_{Gp}\left(  p\right)  \leq\left\langle \nabla d,J\right\rangle _{M}\left(
\pi\left(  p\right)  \right)  ,
\]
for all $p\in\pi^{-1}\left(  \Lambda\right)  $, where $H_{Gp}$ is given by
(\ref{HG}).
\end{lemma}

\begin{proof}
As $\overrightarrow{H}_{G}$ and $\nu$ are the horizontal lift of $J$ and
$\nabla d$ respectively, we have $J\left(  \pi\left(  p\right)  \right)
=d\pi_{p}\left(  \overrightarrow{H}_{Gp}\left(  p\right)  \right)  $ and
$\nabla d\left(  \pi\left(  p\right)  \right)  =d\pi_{p}\left(  \nu\left(
p\right)  \right)  $. Since $\pi$ is a Riemannian submersion,
\[
d\pi_{p}|_{\left[  T_{p}\mathbb{R}^{n+1}\right]  ^{h}}:\left[  T_{p}%
\mathbb{R}^{n+1}\right]  ^{h}\longrightarrow T_{\pi\left(  p\right)  }M
\]
is an isometry, where $\left[  T_{p}\mathbb{R}^{n+1}\right]  ^{h}$ means the
horizontal vector space relatively to $Gp$ at $p$ and, from this, we have
(\ref{hl}). In particular, $\left\langle \nu,\overrightarrow{H}_{Gp}%
\right\rangle $ is constant along $Gp$. Note that $\left\vert \nu\right\vert
=1$ since $\left\vert \nabla d\right\vert =1$. Thus
\[
-H_{Gp}\left(  p\right)  =-\left\vert \overrightarrow{H}_{Gp}\left(  p\right)
\right\vert \leq\left\langle \nu,\overrightarrow{H}_{Gp}\right\rangle \left(
p\right)
\]
and the result follows.
\end{proof}

\begin{proposition}
\label{P_ginf}Let $G,$ $\pi$ and $M=(\mathbb{R}^{n},\left\langle
,\right\rangle _{G})$ as defined in (\ref{G}), (\ref{hp}) and (\ref{M}),
respectively, and assume $n\geq3$ and $\lambda\neq0$. Let $o$ be an arbitrary
but fixed point of $M$ and let $\Lambda=M-\overline{B_{r}(o)}$, where
$B_{r}(o)$ is the open geodesic ball of $M$ of radius $r$ centered at $o$. Let
$b\in\mathbb{R}$ satisfying $b>C$, where $C$ is given by (\ref{C}). Consider
$\psi:[0,\infty)\longrightarrow\mathbb{R}$ given by%
\begin{equation}
\psi\left(  t\right)  =\frac{1}{b}\cosh^{-1}\left(  1+bt\right)
\label{psi_ginf}%
\end{equation}
and $\Lambda_{0}:=\left\{  q\in\Lambda;d\left(  q\right)  \leq t_{0}\right\}
$, where
\begin{equation}
t_{0}=\frac{b-C}{bC}. \label{t0_ginf}%
\end{equation}
Then $w=\psi\circ d:\Lambda_{0}\rightarrow\mathbb{R}$ is such that $w\in
C^{2}\left(  \Lambda_{0}\right)  \cap C^{0}\left(  \overline{\Lambda}%
_{0}\right)  $, $w|_{\partial\Lambda}=0$, $w>0$ on $\Lambda_{0}$,
\[
\underset{d\left(  q\right)  \rightarrow0}{\lim}\left\vert \nabla w\left(
q\right)  \right\vert =+\infty
\]
and $\mathfrak{M}\left(  w\right)  \leq0$ on $\Lambda_{0}$, where
$\mathfrak{M}$ is the operator defined in (\ref{PDM}).
\end{proposition}

\begin{proof}
Let $\varphi\in C^{2}(\left(  0,\infty)\right)  \mathbb{\cap}C^{0}\left(
[0,\infty)\right)  $ to be determined \textit{a posteriori} and consider the
function $w:\Lambda\subset M\longrightarrow\mathbb{R}$ given by $w\left(
q\right)  =\left(  \varphi\circ d\right)  \left(  q\right)  $. Straightforward
calculations give us that, in $\Lambda$,%
\[
div_{M}\left(  \frac{\nabla w}{\sqrt{1+\left\vert \nabla w\right\vert ^{2}}%
}\right)  =g\left(  d\right)  \Delta d+g^{\prime}\left(  d\right)
\]
and
\[
\frac{1}{\sqrt{1+\left\vert \nabla w\right\vert ^{2}}}\left\langle \nabla
w,J\right\rangle _{M}=g\left(  d\right)  \left\langle \nabla d,J\right\rangle
_{M},
\]
where
\begin{equation}
g\left(  d\right)  :=\frac{\varphi^{\prime}\left(  d\right)  }{\sqrt{1+\left[
\varphi^{\prime}\left(  d\right)  \right]  ^{2}}}, \label{g}%
\end{equation}
$\Delta$ is the Laplacian in $M$ and $"^{\prime}"$ means $\frac{\partial
}{\partial d}$. Thus, $\mathfrak{M}\left(  w\right)  \leq0$ in $\Lambda$ if
and only if
\begin{equation}
g\left(  d\right)  \Delta d+g^{\prime}\left(  d\right)  -g\left(  d\right)
\left\langle \nabla d,J\right\rangle _{M}\leq0\text{ in }\Lambda. \label{1}%
\end{equation}
From the Laplacian's Comparison Theorem, since $Ric_{M}\geq0$ and $\dim M=n$,
we have
\begin{equation}
\Delta d\left(  q\right)  \leq\frac{n-1}{d\left(  q\right)  +r}\leq\frac
{n-1}{r}\text{, }q\in\Lambda. \label{Lap}%
\end{equation}
Now, we assume that our $\varphi$ satisfies $\varphi\left(  0\right)  =0$ and
$\varphi^{\prime}\left(  d\right)  >0$ for $d>0$ (consequently, $\varphi
\left(  d\right)  >0$ for $d>0$). From (\ref{g}), it follows that $g\left(
d\right)  >0$ for $d>0$ and, from (\ref{Lap}), we conclude that if
\begin{equation}
\frac{\left(  n-1\right)  g\left(  d\right)  }{r}+g^{\prime}\left(  d\right)
-g\left(  d\right)  \left\langle \nabla d,J\right\rangle _{M}\leq0\text{ in
}\Lambda\label{2}%
\end{equation}
then we have (\ref{1}). From Lemma \ref{nv} and (\ref{sHG}), we see that
\[
-\frac{\left\vert \lambda\right\vert }{2\left\vert a\right\vert }\leq
-H_{Gp}\left(  p\right)  \leq\left\langle \nabla d,J\right\rangle _{M}\circ
\pi\left(  p\right)
\]
and, then, if%
\begin{equation}
\frac{g^{\prime}\left(  d\right)  }{g\left(  d\right)  }\leq-C\text{ in
}\Lambda, \label{3}%
\end{equation}
where $C$ is given by (\ref{C}), then we have (\ref{2}). From (\ref{g}), we
see that (\ref{3}) is equivalent to%
\begin{equation}
\frac{\varphi^{\prime\prime}\left(  d\right)  }{\varphi^{\prime}\left(
d\right)  }\leq-C\left(  1+\left[  \psi^{\prime}\left(  d\right)  \right]
^{2}\right)  \text{ in }\Lambda. \label{4}%
\end{equation}
We will assume from now on that our $\varphi$ satisfies%
\[
\underset{d\rightarrow0}{\lim}\varphi^{\prime}\left(  d\right)  =+\infty.
\]
A function $\varphi$ which satisfies all the requirements demanded until now
is given by%
\[
\varphi\left(  t\right)  =\alpha\cosh^{-1}\left(  1+bt\right)
\]
$t\geq0$, with $\alpha$, $b$ positive constants to be determinate, where,
here, $t=d\left(  q\right)  $, $q\in\Lambda$. Assuming such one $\varphi$, we
see that (\ref{4}) is equivalent to%
\begin{equation}
\frac{-b\left(  1+bt\right)  }{\left(  \left(  1+bt\right)  ^{2}-1+\alpha
^{2}b^{2}\right)  }\leq-C. \label{5}%
\end{equation}
We assume $\alpha,b$ such that $\alpha b=1$. Thus, we have (\ref{5}) if%
\[
t\leq\frac{b-C}{bC}\text{, }C<b.
\]
Thus, assuming $b>C$, setting%
\[
t_{0}:=\frac{b-C}{bC},
\]
considering the neighborhood of $\partial\Lambda$ in $\Lambda$ given by
\[
\Lambda_{0}:=\left\{  q\in\Lambda;d\left(  q\right)  \leq t_{0}\right\}  ,
\]
and the function
\[
\psi\left(  t\right)  =\frac{1}{b}\cosh^{-1}\left(  1+bt\right)  \text{,
}t\geq0\text{,}%
\]
we have that
\begin{equation}
w\left(  q\right)  =\psi\circ d\left(  q\right)  \text{, }q\in\Lambda_{0},
\label{wL0}%
\end{equation}
satisfies%
\[
\mathfrak{M}\left(  w\right)  \leq0
\]
in $\Lambda_{0}$ (note that $\exp_{\partial\Lambda}:\partial\Lambda
\times\lbrack0,t_{0}]\longrightarrow\Lambda_{0}$, $\exp_{\partial\Lambda
}\left(  p,t\right)  =\exp_{p}t\eta\left(  p\right)  $, where $\eta$ is the
unit vector field normal to $\partial\Lambda$ that points to $\Lambda$, is a
diffeomorphism, since $\Lambda$ is the exterior of a geodesic ball in $M$).
The others conclusion follow directly from the definition of $\psi$.
\end{proof}

\begin{corollary}
\label{mir}Assume the same hypotheses of Proposition \ref{P_ginf}. Let
$\varsigma>C$ be the solution of the equation (\ref{eqmi}), where $C=C\left(
r,n,\lambda,a\right)  $ is given by (\ref{C}). The function $w$ given in
Proposition \ref{P_ginf} with the biggest height at $\partial\Lambda
_{0}-\partial\Lambda$ is obtained taking $b=\varsigma$. In particular, for
such $w$,%
\[
\underset{\overline{\Lambda}_{0}}{\sup}w=\sup_{\partial\Lambda_{0}%
-\partial\Lambda}w=%
\mathcal{L}%
,
\]
where $%
\mathcal{L}%
$ is given by (\ref{L}) and, setting $W:\overline{\Lambda}\subset
M\longrightarrow\mathbb{R}$ given by%
\begin{equation}
W\left(  q\right)  =\left\{
\begin{array}
[c]{c}%
w\left(  q\right)  \text{, if }q\in\Lambda_{0}\\%
\mathcal{L}%
\text{ if }q\in\Lambda-\Lambda_{0}%
\end{array}
\right.  , \label{Wmir}%
\end{equation}
we have $W\in C^{0}\left(  \overline{\Lambda}\right)  $ and radial with
relation to the point $o$, $\mathfrak{M}\left(  W\right)  =0$ on
$\Lambda-\partial\Lambda_{0}$, $W|_{\partial\Lambda}=0$, with%
\[
\underset{d\left(  q\right)  \rightarrow0}{\lim}\left\vert \nabla W\left(
q\right)  \right\vert =+\infty.
\]

\end{corollary}

\begin{proof}
Given $\mu>C$, take $b$ as in Proposition \ref{P_ginf} given by $b=\mu$. The
correspondent $t_{0}$ is%
\[
t_{0}=\frac{\mu-C}{\mu C}%
\]
and, at $t_{0}$, from (\ref{wL0}), we have $\psi\left(  t_{0}\right)
=\mu^{-1}\cosh^{-1}\left(  \mu C^{-1}\right)  $. The function
\[
f\left(  \mu\right)  =\frac{1}{\mu}\cosh^{-1}\left(  \frac{\mu}{C}\right)
\text{, }\mu>C\text{,}%
\]
clearly satisfies $f\left(  \mu\right)  \rightarrow0$ when $\mu\rightarrow C $
and $f\left(  \mu\right)  \rightarrow0$ when $\mu\rightarrow+\infty$ and the
absolute maximal point of $f$ in $\left(  C,\infty\right)  $ is reach at the
point $\varsigma$ which is solution of (\ref{eqmi}) and $f\left(
\varsigma\right)  =\left(  \varsigma^{2}-C^{2}\right)  ^{-1/2}$. The other
conclusions follow from the definition of $W$ and from the fact that $\Lambda$
is the exterior of a geodesic ball of $M$ center at $o$.
\end{proof}

\begin{remark}
\label{R}We observe that if $\lambda=0$, then $G$ is the group of the
translations in $e_{k}$- direction. In this case, we have $M\equiv
\mathbb{R}^{n}$ and the domain $\Lambda$, as in the hypothesis of Proposition
\ref{P_ginf}, is the exterior of the geodesic ball of $\mathbb{R}^{n}$ of
radius $r$. Then $v_{r}$ given by (\ref{ncat}) is a solution of (\ref{PDM}) if
the boundary data is zero. In particular, if $n\geq3$, its height at infinity
is $h\left(  n,r\right)  $, where $h\left(  n,r\right)  $ is given by
(\ref{hro}).
\end{remark}

\begin{proposition}
\label{CA}Let $G,$ $\pi$ and $M=(\mathbb{R}^{n},\left\langle ,\right\rangle
_{G})$, $n\geq2$, as defined in (\ref{G}), (\ref{hp}) and (\ref{M}),
respectively. Let $\Lambda_{\rho}:=M-\overline{\mathfrak{B}_{\rho}\left(
0\right)  }$, where $\mathfrak{B}_{\rho}\left(  0\right)  $ is the open
geodesic ball of $\mathbb{R}^{n}$ of radius $\rho$ centered at origin of
$\mathbb{R}^{n}$. Then $v_{\rho}\in C^{2}\left(  \Lambda_{\rho}\right)  \cap
C^{0}\left(  \Lambda_{\rho}\right)  $ given by (\ref{ncat}) is a non-negative
solution of the Dirichlet (\ref{PDM}) relatively to $\Lambda_{\rho}$ with
$v_{\rho}|_{\partial\Lambda_{\rho}}=0$, which is unbounded if $n=2$ and
satisfies $\left\vert \nabla v_{\rho}\right\vert \circ\pi=\left\vert
\overline{\nabla}v_{\rho}\right\vert $,%
\begin{equation}
\underset{d\left(  q\right)  \rightarrow0}{\lim}\left\vert \nabla v_{\rho
}\left(  q\right)  \right\vert =\infty\text{, }\underset{d\left(  q\right)
\rightarrow\infty}{\lim}\left\vert \nabla v_{\rho}\left(  q\right)
\right\vert =0 \label{hvr0}%
\end{equation}
and
\begin{equation}
\underset{d\left(  q\right)  \rightarrow\infty}{\lim}v_{\rho}\left(  q\right)
=h\left(  n,\rho\right)  \text{ if }n\geq3 \label{hvr}%
\end{equation}
where $d=d_{M}\left(  .,\partial\Lambda_{\rho}\right)  $, $h\left(
n,\rho\right)  $ is given by (\ref{hro}).
\end{proposition}

\begin{proof}
Let $v_{\rho}:\mathbb{R}^{n}-\mathfrak{B}_{\rho}\left(  0\right)
\rightarrow\mathbb{R}$ be the function given by (\ref{ncat}), $\mathbb{R}%
^{n}\equiv\nolinebreak\left\{  x_{k}=0\right\}  $. As the graph of $v_{\rho}$
is a minimal graph in $\mathbb{R}^{n+1}$ it follows that the graph of
\[
u_{\rho}\left(  x_{1},...,x_{n+1}\right)  :=v_{\rho}\left(  x\right)  ,\text{
}x=%
{\textstyle\sum\limits_{i\neq k}}
x_{i}e_{i}\in\mathbb{R}^{n}-\mathfrak{B}_{\rho}\left(  0\right)  ,
\]
is a minimal graph in $\mathbb{R}^{n+2}$. In particular $\left\vert
\overline{\nabla}u_{\rho}\left(  x_{1},...,x_{n+1}\right)  \right\vert
=\left\vert \overline{\nabla}v_{\rho}\left(  x\right)  \right\vert $. Note
that, setting $\Omega_{\rho}\subset\mathbb{R}^{n+1}$ the domain of $u_{\rho}$,
we have $\Omega_{\rho}$ a $G$-invariant domain such that its helicoidal
projection $\pi\left(  \Omega\right)  $ on $\mathbb{R}^{n}$ coincides with the
image of the its orthogonal projection on $\mathbb{R}^{n}$ in this case. From
(\ref{hp}) we see that $\left\vert x\right\vert =\left\vert \pi\left(
x_{1},...,x_{n+1}\right)  \right\vert $. As $v_{\rho}$ is radial, it follows
that
\[
u_{\rho}\left(  x_{1},...,x_{n+1}\right)  =\left(  v_{\rho}\circ\pi\right)
\left(  x_{1},...,x_{n+1}\right)
\]
and, then, $u_{\rho}$ is a $G$-invariant function with $u_{\rho}%
|_{\partial\Omega_{\rho}}=0$. It follows from Proposition 3 of \cite{RT} that
$v_{\rho}\in C^{2}\left(  \Lambda_{\rho}\right)  \cap C^{0}\left(
\Lambda_{\rho}\right)  $ is a solution of the Dirichlet problem (\ref{PDM})
relatively to $\Lambda_{\rho}$ with $v_{\rho}|_{\partial\Lambda_{\rho}}=0$
and, moreover,
\[
\left\vert \nabla v_{\rho}\right\vert \circ\pi=\left\vert \overline{\nabla
}u_{\rho}\right\vert .
\]
The other conclusions follow immediately from the definition of $u_{\rho}$ and
from the properties satisfied by $v_{\rho}$, taking into account that
$d\left(  q\right)  \rightarrow0$ ($d\left(  q\right)  \rightarrow\infty$) if
only if $d_{E}\left(  q,\partial\Lambda_{\rho}\right)  \rightarrow0$
($d_{E}\left(  q,\partial\Lambda_{\rho}\right)  \rightarrow\infty$).
\end{proof}

\section{Proof of the main results}

The proof of Theorem \ref{MT} and Corollary \ref{CMT} follow directly from the
results of this section, by using Proposition 3 of \cite{RT}.

\begin{proposition}
\label{RP}Let $G,$ $\pi$ and $M=(\mathbb{R}^{n},\left\langle ,\right\rangle
_{G})$ as defined in (\ref{G}), (\ref{hp}) and (\ref{M}), respectively. Let
$U$ be an exterior $C^{2,\alpha}$-domain in $M$, $U=\pi\left(  \Omega\right)
$, where $\Omega\subset\mathbb{R}^{n+1}$ is a $G$-invariant $C^{2,\alpha}%
$-domain and let $d=d_{M}\left(  .,\partial U\right)  $. Let $\varrho>0$ be
the radius of the smallest geodesic ball of $\mathbb{R}^{n}$ which contain
$M-U$, centered at origin of $\mathbb{R}^{n}$ if $\lambda\neq0$. Given
$s\geq0$, there is a non-negative solution $\vartheta_{s}\in C^{2,\alpha
}\left(  \overline{U}\right)  $ of the Dirichlet problem (\ref{PDM}) with
$\vartheta_{s}|_{\partial U}=0$, such that
\[
\sup_{\overline{U}}\left\vert \nabla\vartheta_{s}\right\vert =\sup_{\partial
U}\left\vert \nabla\vartheta_{s}\right\vert =s,
\]
$\vartheta_{s}$ is unbounded if $n=2$ and $s\neq0$ and, if $n\geq3$, either%
\begin{equation}
\underset{\overline{U}}{\sup}\vartheta_{s}\leq h\left(  n,\varrho\right)
\label{vvd}%
\end{equation}
or there is a complete, non-compact, properly embedded $\left(  n-1\right)
$-dimensional submanifold $\Sigma$ of $M$, $\Sigma\subset U$, such that%
\begin{equation}
\underset{d\left(  q\right)  \rightarrow+\infty}{\lim}\vartheta_{s}\left\vert
_{\Sigma}\right.  \left(  q\right)  =h\left(  n,\varrho\right)  , \label{vvd2}%
\end{equation}
where $h\left(  n,\varrho\right)  $ is given by (\ref{hro}). Moreover,
\[
\underset{d\left(  q\right)  \rightarrow\infty}{\lim}\left\vert \nabla
\vartheta_{s}\left(  q\right)  \right\vert =0
\]
if $\lambda=0$ or $s=0$ or, if $\lambda\neq0$, $3\leq n\leq6$ and
$\vartheta_{s}$ is bounded.
\end{proposition}

\begin{proof}
As mentioned in Remark \ref{R}, if $\lambda=0$ then $M\equiv\mathbb{R}^{n}$
and the result is already contemplate in Theorem 1 of \cite{ABR} for $n\geq3$
and in Theorem 1 of \cite{RTa} if $n=2$. The case $s=0$ is trivial. Assume
then $s>0$ and $\lambda\neq0$.\newline Let $\rho>0$ be such that $\partial
U\subset\mathfrak{B}_{\rho}\left(  0\right)  $, where $\mathfrak{B}_{\rho
}:=\mathfrak{B}_{\rho}\left(  0\right)  $ is the open geodesic ball of
$\mathbb{R}^{n}$, centered at origin of $\mathbb{R}^{n}$ and of radius $\rho$.
Let $\Lambda_{\rho}=M-\overline{\mathfrak{B}_{\rho}\left(  0\right)  }$. From
Proposition \ref{CA} , there is $v_{\rho}\in C^{\infty}\left(  \Lambda_{\rho
}\right)  $ solution of (\ref{PDM}) relatively to $\Lambda_{\rho}$, with
$v_{\rho}|_{\partial\Lambda_{\rho}}=0$, satisfying what is stated in
(\ref{hvr0}). Since%
\begin{equation}
\underset{d_{M}\left(  q,\partial\Lambda_{\rho}\right)  \rightarrow\infty
}{\lim}\left\vert \nabla v_{\rho}\left(  q\right)  \right\vert =0, \label{gt0}%
\end{equation}
we can choose $k>\rho$ such that
\begin{equation}
\left\vert \nabla v_{\rho}\right\vert _{\partial\mathfrak{B}_{k}\left(
0\right)  }\leq\frac{s}{2}, \label{vas}%
\end{equation}
where $\mathfrak{B}_{k}\left(  0\right)  $ is the open geodesic ball of
$\mathbb{R}^{n}$ centered at origin and of radius $k$. Let $U_{k}%
=\mathfrak{B}_{k}\left(  0\right)  \cap U$ and define%
\begin{equation}
T_{k}:=\left\{
\begin{array}
[c]{c}%
t\geq0\ ;\exists~w_{t}\in C^{2,\alpha}\left(  \overline{U}_{k}\right)
;\mathfrak{M}\left(  w_{t}\right)  =0\\
\sup_{\overline{U}_{k}}\left\vert \nabla w_{t}\right\vert \leq s,~w_{t}%
|_{\partial U}=0,w_{t}|_{\partial\mathfrak{B}_{k}\left(  0\right)  }=t
\end{array}
\right\}  . \label{Tk}%
\end{equation}
Note that the constant function $w_{0}\equiv0$ on $\overline{U}_{k}$ satisfies
all the condition in (\ref{Tk}), then $T_{k}\neq\emptyset$. Moreover, $\sup
T_{k}<\infty$ since
\[
\sup_{\overline{U}_{k}}\left\vert \nabla w_{t}\right\vert \leq s
\]
for all $t\in T_{k}$. Now, since the maximum principle and comparison
principle are applicable relatively to the operator $\mathfrak{M}$, we can use
the same approach used in the proof of Theorem 1 of \cite{ABR} to show that
$t_{k}:=\sup T_{k}\in T_{k}$,%
\[
\sup_{\overline{U}_{k}}\left\Vert \nabla w_{t_{k}}\right\Vert =s,\text{ }%
\sup_{\partial\mathfrak{B}_{k}}\left\vert \nabla w_{t_{k}}\right\vert \leq
s/2
\]
and, since $\partial U_{k}=\partial U\cup\partial\mathfrak{B}_{k}\left(
0\right)  $, to conclude that
\begin{equation}
\sup_{\overline{U}_{k}}\left\vert \nabla w_{t_{k}}\right\vert =\sup_{\partial
U}\left\vert \nabla w_{t_{k}}\right\vert =s\text{.} \label{wtk}%
\end{equation}
The proof of these facts follow essentially the same steps of the
aforementioned theorem (see p. 3067 and 3068 of \cite{ABR}), and then we will
not do it here. Now, taking $k\rightarrow\infty$ and from diagonal method, we
obtain a subsequence of $\left(  w_{t_{k}}\right)  $ which converges uniformly
in the $C^{2}$ norm in compact subsets of $\overline{U}$ to a function
$\vartheta_{s}\in C^{2,\alpha}\left(  \overline{U}\right)  $ satisfying
$\mathfrak{M}\left(  \vartheta_{s}\right)  =0$ in $U$, $\vartheta
_{s}|_{\partial U}=0$, which is non-negative and such that $\sup_{\overline
{U}}\left\vert \nabla\vartheta_{s}\right\vert =\sup_{\partial U}\left\vert
\nabla\vartheta_{s}\right\vert =s$. In particular, from regularity elliptic
PDE theory (\cite{GT}), we have $\vartheta_{s}\in C^{\infty}\left(  U\right)
$.\newline We will show now that if $n=2$ then $\vartheta_{s}$ is unbounded
and, if $n\geq3$, we have either (\ref{vvd}) or (\ref{vvd2}).\newline Let
$\varrho>0$ be the radius of the smallest open geodesic ball of $\mathbb{R}%
^{n}$ which contain $M-U$, centered at origin of $\mathbb{R}^{n}$ and denote
such ball by $\mathfrak{B}_{\varrho}\left(  0\right)  $. We have $\partial
U\subset\overline{\mathfrak{B}_{\varrho}\left(  0\right)  }$ and we can
conclude that $\partial U\cap\partial\mathfrak{B}_{\varrho}\neq\emptyset$. Let
$\Lambda_{\varrho}=M-\overline{\mathfrak{B}_{\varrho}\left(  0\right)  }$.
From Proposition \ref{CA} , there is $v_{\varrho}\in C^{\infty}\left(
\Lambda_{\varrho}\right)  $, $v_{\rho}|_{\partial\Lambda_{\varrho}}=0$,
solution of (\ref{PDM}), sat\nolinebreak isfying what is stated in
(\ref{hvr0}) and (\ref{hvr}) relatively to $\Lambda_{\varrho}$ if $n\geq3$.
Otherwise, if $n=2$, $v_{\varrho}$ is unbounded and satisfies the equalities
in (\ref{hvr0}).\newline Let $q_{0}\in\partial U\cap\partial\mathfrak{B}%
_{\varrho}$. Since
\[
\underset{d_{M}(q,\partial\Lambda_{\varrho})\rightarrow0}{\lim}\left\vert
\nabla v_{\varrho}\left(  q\right)  \right\vert =+\infty\text{, }%
\underset{\overline{U}}{\sup}\left\vert \nabla\vartheta_{s}\right\vert
=\underset{\partial U}{\sup}\left\vert \nabla\vartheta_{s}\right\vert
=s<+\infty
\]
and $v_{\varrho}\left(  q_{0}\right)  =\vartheta_{s}\left(  q_{0}\right)  =0$,
it follows that there is an open set $V_{q_{0}}$ in $U\cap\Lambda_{\varrho}$,
with $q_{0}\in\partial V_{q_{0}}$, such that $\vartheta_{s}<v_{\varrho}$ in
$V_{q_{0}}$. We claim that $V_{q_{0}}$ is unbounded. Suppose that $V_{q_{0}}$
is bounded. Since $\vartheta_{s}|_{\partial V_{q_{0}}}=v_{\varrho}|_{\partial
V_{q_{0}}}$, it follows that $\vartheta_{s}|_{\overline{V}_{q_{0}}}$ and
$v_{\varrho}|_{\overline{V}_{q_{0}}}$ are distinct solutions to the Dirichlet
problem%
\begin{equation}
\left\{
\begin{array}
[c]{c}%
\mathfrak{M}(f)=0\text{ in }V_{q_{0}}\text{, }f\in C^{2}\left(  V_{q_{0}%
}\right)  \cap C^{0}(\overline{V}_{q_{0}})\\
f|_{\partial V_{q_{0}}}=\vartheta_{s}|_{\partial V_{q_{0}}}%
\end{array}
\right.  , \label{DP3}%
\end{equation}
a contradiction, since a solution of (\ref{DP3}) is unique if the domain is
bounded. It follows that $V_{q_{0}}$ is unbounded. Note that we can have two
possibility for $\partial V_{q_{0}}$: either $\partial V_{q_{0}}$ is bounded
(in this case $V_{q_{0}}=\Lambda_{\varrho}$) or $\partial V_{q_{0}}$ is
unbounded (in this case, setting $\Sigma=\partial V_{q_{0}}$, we have
$\Sigma\subset\overline{U}\cap\overline{\Lambda}_{\varrho}$ a complete
$\left(  n-1\right)  $-dimensional manifold of $M$).\newline Assume first that
$n\geq3$. In this case $v_{\varrho}$ is bounded. If $\partial V_{q_{0}}$ is
bounded, as in this case $V_{q_{0}}=\Lambda_{\varrho}$, this means that
$\vartheta_{s}<v_{\varrho}$ on $\Lambda_{\varrho}$ and so we have (\ref{vvd}).
If $\partial V_{q_{0}}$ is unbounded, as $\vartheta_{s}|_{\partial V_{q_{0}}%
}=v_{\varrho}|_{\partial V_{q_{0}}}$, we conclude that%
\[
\underset{d\left(  q\right)  \rightarrow+\infty}{\lim}\vartheta_{s}|_{\partial
V_{q_{0}}}\left(  q\right)  =h\left(  n,\varrho\right)  .
\]
Assume now that $n=2$. Note first that on $\Lambda_{\varrho}\subset
\mathbb{R}^{2}$ it is well know that, for all $s>0$, there is a half catenoid
$v_{\varrho,s}\in C^{2}\left(  \overline{\Lambda}_{\varrho}\right)  $ which is
unbounded (logarithmic growth), which satisfies $v_{\varrho,s}|_{\partial
\Lambda_{\varrho}}=0$ and $\left\vert \overline{\nabla}v_{\varrho
,s}\right\vert =s$ on $\partial\Lambda_{\varrho}$ ($=\partial\mathfrak{B}%
_{\varrho}$). The same arguments used in Proposition \ref{CA} give us that
$v_{\varrho,s}$ is solution for the Dirichlet problem (\ref{PDM}) for zero
boundary data relatively to $\Lambda_{\varrho}$ and satisfies $\left\vert
\nabla v_{\varrho,s}\right\vert =s$ on $\partial\Lambda_{\varrho}$, since
$\mathfrak{B}_{\varrho}$ is centered at origin of $\mathbb{R}^{2}$. As
$\sup_{\partial U}\left\vert \nabla\vartheta_{s}\right\vert =s$ and $\partial
U\cap\partial\mathfrak{B}_{\varrho}\neq\emptyset$, there is $0<s^{\prime}<s$
such that $v_{\varrho,s^{\prime}}\in\nolinebreak C^{2}\left(  \overline
{\Lambda}_{\varrho}\right)  $ as described above satisfies $v_{\varrho
,s^{\prime}}<\vartheta_{s}$ in some open set $V\subset U\cap\Lambda_{\varrho}%
$. The same arguments used before to prove that $V_{q_{0}}$ is unbounded give
us that $V$ is unbounded and we see that if $\partial V$ is bounded then
$V=\Lambda_{\varrho}$. Thus, if $\partial V$ is bounded, we have
$v_{\varrho,s^{\prime}}<\vartheta_{s}$ on $\Lambda_{\varrho}$ and then,
$\vartheta_{s}$ is unbounded. If $\partial V$ is unbounded, as $v_{\varrho
,s^{\prime}}|_{\partial V}=\vartheta_{s}|_{\partial V},$ we have%
\[
\underset{d\left(  q\right)  \rightarrow+\infty}{\lim}\vartheta_{s}|_{\partial
V}\left(  q\right)  =+\infty
\]
since
\[
\underset{d\left(  q\right)  \rightarrow+\infty}{\lim}v_{\varrho,s^{\prime}%
}|_{\partial V}\left(  q\right)  =+\infty.
\]
and, therefore, $\vartheta_{s}$ is unbounded.\newline Now we will prove the
last affirmation of the proposition.\newline As $\vartheta_{s}\in C^{2,\alpha
}\left(  \overline{U}\right)  $ is a solution of the Dirichlet problem
(\ref{PDM}) with $\vartheta_{s}|_{\partial U}=0$, it follows from Proposition
3 of \cite{RT} that $u=\vartheta_{s}\circ\pi\in C^{2,\alpha}\left(
\overline{\Omega}\right)  $ satisfies $\mathcal{M}\left(  u\right)  =0$ in
$\overline{\Omega}$ with $u|_{\partial\Omega}=0$ and $\left\vert
\overline{\nabla}u\left(  p\right)  \right\vert =\left\vert \nabla\left(
\vartheta_{s}\circ\pi\right)  \left(  p\right)  \right\vert $, $p\in\Omega$,
where $\mathcal{M}$ is the operator defined in (\ref{PD}). Note that we have
necessarily $\Omega$ unbounded with $d_{E}\left(  p\right)  \rightarrow\infty$
if only if $d\left(  \pi\left(  p\right)  \right)  \rightarrow\infty$, where
$d_{E}$ is the Euclidean distance in $\mathbb{R}^{n+1}$ to $\partial\Omega$.
Suppose that
\[
\underset{d_{E}\left(  p\right)  \rightarrow\infty}{\lim}\left\vert
\overline{\nabla}u\left(  p\right)  \right\vert \neq0.
\]
Then there is $\varepsilon>0$ and a sequence $\left(  p_{n}\right)  $ in
$\Omega$, with $d_{E}\left(  p_{n}\right)  \rightarrow\infty$ when
$n\rightarrow\infty$ such that $\left\vert \overline{\nabla}u\left(
p_{n}\right)  \right\vert \geq\varepsilon$ for all $n$ large enough, $n\geq
n_{0}$. For each $n\in\mathbb{N}$, define
\[
\Omega_{n}=\left\{  p\in\mathbb{R}^{n+1};p+p_{n}\in\Omega\right\}
\]
and consider the sequence of functions $u_{n}:\Omega_{n}\subset\mathbb{R}%
^{n+1}\longrightarrow\mathbb{R}$ given by $u_{n}\left(  p\right)  =u\left(
p+p_{n}\right)  $. Note that $0\in\Omega_{n}$ for all $n$ since $\left(
p_{n}\right)  \subset\Omega$. Also%
\[
\mathbb{R}^{n+1}=\bigcup\limits_{n\in\mathbb{N}}\Omega_{n}.
\]
Indeed, given $w\in\mathbb{R}^{n+1}$, if the sequence $(w+p_{n})_{n}$ were
contained in $\mathbb{R}^{n+1}-\Omega$, since $\pi\left(  \mathbb{R}%
^{n+1}-\Omega\right)  $ is compact, we would have $d_{E}\left(  w+p_{n}%
\right)  \leq R$ for all $n$, for some $R>0$, a contradiction, since
$d_{E}\left(  p_{n}\right)  \rightarrow\infty$. It follows that, as
$u_{n}\left(  0\right)  =u\left(  p_{n}\right)  $, for all $n\geq n_{0}$ we
have $\left\vert \overline{\nabla}u_{n}\left(  0\right)  \right\vert
\geq\varepsilon$. Note that $\left(  u_{n}\right)  $ is uniformly bounded
since, by hypothesis, $n\geq3$ and $\vartheta_{s}\ $ is bounded. Then $\left(
u_{n}\right)  $ has a subsequence $\left(  u_{n_{k}}\right)  $ which converges
uniformly on compact subsets of $\mathbb{R}^{n+1}$ to a bounded function
$\widetilde{u}$ defined on the whole $\mathbb{R}^{n+1}$ which satisfies
$\mathcal{M}\left(  \widetilde{u}\right)  =0$. Assume that dimension of $M$
satisfies $3\leq n\leq6$. Then $\Omega\subset\mathbb{R}^{m}$, $4\leq m\leq7$.
From Bersntein Theorem extended to $\mathbb{R}^{m}$, $2\leq m\leq7$, by Simons
(\cite{S}) - which is false for $m\geq8$ (\cite{BGG}) - it follows that
$\widetilde{u}$ has to be constant. Therefore, we cannot have $\left\vert
\overline{\nabla}u_{n_{k}}\left(  0\right)  \right\vert \geq\varepsilon$ for
all $n_{k}\geq n_{0}$, a contradiction. Hence
\[
\underset{d_{E}\left(  p\right)  \rightarrow\infty}{\lim}\left\vert
\overline{\nabla}u\left(  p\right)  \right\vert =0.
\]
From Proposition 3 of \cite{RT} we have $\left\vert \overline{\nabla
}u\right\vert =\left\vert \nabla\vartheta_{s}\right\vert \circ\pi$ and,
therefore
\[
\underset{d\left(  q\right)  \rightarrow\infty}{\lim}\left\vert \nabla
\vartheta_{s}\left(  q\right)  \right\vert =0.
\]

\end{proof}

\begin{remark}
If $\lambda=0$ and $n=2$, it was proved in \cite{RTa} that, setting%
\begin{equation}
\vartheta_{\infty}\left(  p\right)  :=\underset{s\rightarrow\infty}{\lim
}\vartheta_{s}\left(  p\right)  ,p\in U\text{,} \label{vinf}%
\end{equation}
$\vartheta_{\infty}\in C^{2}\left(  U\right)  \cap C^{0}\left(  \overline
{U}\right)  $ and is an unbounded solution of the Dirichlet problem
(\ref{PDM}) with $\vartheta_{\infty}|_{\partial U}=0$, which satisfies
\[
\underset{d_{E}\left(  p\right)  \rightarrow\infty}{\lim}\left\vert
\overline{\nabla}u\left(  p\right)  \right\vert =0.
\]
If $\lambda=0$ and $n\geq3$, it was proved in \cite{ABR} that $\vartheta
_{\infty}$, as defined in (\ref{vinf}), is in $C^{2}\left(  U\right)  $, is a
bounded solution of the Dirichlet problem (\ref{PDM}) and its graph is
contained in a $C^{1,1}$-manifold $\Upsilon\subset\overline{U}\times
\mathbb{R}$ such that $\partial\Upsilon=\partial\Omega$.
\end{remark}

\begin{proposition}
\label{CRP}Let $G,$ $\pi$ and $M=(\mathbb{R}^{n},\left\langle ,\right\rangle
_{G})$, $n\geq3$, as defined in (\ref{G}), (\ref{hp}) and (\ref{M}),
respectively. Let $U$ be an exterior $C^{2,\alpha}$-domain in $M$,
$U=\nolinebreak\pi\left(  \Omega\right)  $, where $\Omega\subset
\mathbb{R}^{n+1}$ is a $G$-invariant $C^{2,\alpha}$-domain and let
$d=d_{M}\left(  .,\partial U\right)  $. Assume that $M-U$ satisfies the
interior sphere condition of radius $r>0$. Let $%
\mathcal{L}%
=%
\mathcal{L}%
\left(  r,n,\lambda\,a\right)  $ be as given in (\ref{L}), where $\lambda$ and
$a$ are given in the definition of $G$. Then, given $c\in\lbrack0,%
\mathcal{L}%
]$, there is $w_{c}\in C^{2}\left(  U\right)  \cap C^{0}\left(  \overline
{U}\right)  $ solution of the Dirichlet problem (\ref{PDM}) relatively to $U$,
with $w_{c}|_{\partial U}=0$, which satisfies
\[
\underset{d\left(  q\right)  \rightarrow\infty}{\lim w_{c}\left(  q\right)
}=c.
\]
In particular, if $c\in\lbrack0,%
\mathcal{L}%
)$, then $w_{c}\in C^{2,\alpha}\left(  \overline{U}\right)  $.
\end{proposition}

\begin{proof}
If $\lambda=0$ the result is already contemplate in Theorem 1 of \cite{ABR}.
Assume $\lambda\neq0$. Given $c\in\lbrack0,%
\mathcal{L}%
)$, define
\begin{equation}
\digamma=\left\{
\begin{array}
[c]{c}%
f\in C^{0}\left(  \overline{U}\right)  ;\text{ }f\text{ is subsolution
relative to }\mathfrak{M}\\
f|_{\partial U}=0\text{ and }\underset{d\left(  q\right)  \rightarrow\infty
}{\lim}\sup f\left(  q\right)  \leq c
\end{array}
\right\}  . \label{per}%
\end{equation}
Note that $\digamma\neq\emptyset$ since $f_{0}\equiv0\in$ $\digamma$. From
comparison principle we have $f\leq c$ for all $f\in\digamma$. From Perron
method applied relatively to operator $\mathfrak{M}$ (\cite{GT} - Section 2.8,
\cite{RT} - Section 3), we conclude that
\[
w_{c}\left(  q\right)  :=\sup\left\{  f\left(  q\right)  ;\text{ }f\in
\digamma\right\}  \text{, }q\in\overline{U},
\]
is in $C^{\infty}\left(  U\right)  $ and satisfies$\mathcal{\ }\mathfrak{M}%
\left(  w_{c}\right)  =0$ in $U$. We will show now that
\begin{equation}
\lim_{d\left(  q\right)  \rightarrow\infty}w_{c}\left(  q\right)  =c.
\label{wc}%
\end{equation}
Consider $\alpha>0$ such that $\mathfrak{B}_{\alpha}\left(  0\right)  $, the
geodesic ball of $\mathbb{R}^{n}$ center at origin of $\mathbb{R}^{n}$, of
radius $\alpha>0$, contain $M-U$ and be such that $v_{\alpha}\left(
\infty\right)  >c$, where $v_{\alpha}$ is as defined in (\ref{ncat}) (note
that $v_{\alpha}\left(  \infty\right)  =\alpha h\left(  n,1\right)  $ and
$h\left(  n,1\right)  >0$). Define now $f\in C^{0}\left(  \overline{U}\right)
$ by%
\[
f\left(  q\right)  =\left\{
\begin{array}
[c]{l}%
0\text{ if }q\in\overline{U}\cap\mathfrak{B}_{\alpha}\left(  0\right) \\
\max\{0,v_{\alpha}\left(  q\right)  -\left(  v_{\alpha}\left(  \infty\right)
-c\right)  \}\text{, if }q\in M-\mathfrak{B}_{\alpha}\left(  0\right)  .
\end{array}
\right.
\]
From Proposition \ref{CA} we have $f$ a non-negative (generalized) subsolution
to the Dirichlet problem (\ref{PDM}) (for zero boundary data) relatively to
$U$, which satisfies
\begin{equation}
\underset{d\left(  q\right)  \rightarrow\infty}{\lim}f\left(  q\right)  =c.
\label{win}%
\end{equation}
It follows that $f\in\digamma$ and that $f\leq w_{c}\leq c$. Then we have
(\ref{wc}).\newline Given $q_{0}\in\partial U$, by hypothesis there is a
geodesic open ball of $M$, say $B_{r}$, of radius $r>0$, contained in $M-U$
and such that $\partial B_{r}$ is tangent to $\partial U$ ($=\partial\left(
M-U\right)  $) at $q_{0}$. From Corollary \ref{mir}, there is $W\in
C^{0}\left(  M-B_{r}\right)  $ a (gen\nolinebreak eralized) supersolution
relatively to the operator $\mathfrak{M}$ on $M-B_{r}$ such that $c\leq
W\left(  \infty\right)  =%
\mathcal{L}%
$, with $W|_{\partial B_{r}}=0$, which is $C^{1}$ in a neighborhood of
$\partial B_{r}$ in $M-B_{r}$ and such that
\[
\underset{d_{M}\left(  q,\partial B_{r}\right)  \rightarrow0}{\lim}\left\vert
\nabla W\left(  q\right)  \right\vert =+\infty.
\]
From the comparison principle, since $\overline{U}\subset M-B_{r}$, it follows
that on $\overline{U}$, we have $0\leq w_{c}\leq W$. As $q_{0}$ is arbitrary,
we conclude that $w_{c}\in C^{0}\left(  \overline{U}\right)  $ with
$w_{c}|_{\partial U}=0$.\newline Assume that $0\leq c<%
\mathcal{L}%
$. Let $\delta=\left(  c+%
\mathcal{L}%
\right)  /2$. Let $L\left(  \sigma\right)  :=%
\mathcal{L}%
\left(  \sigma,n,\lambda,a\right)  $, $\sigma\in(0,r]$, where $%
\mathcal{L}%
$ is given by (\ref{L}). Since $L\in C^{0}(0,r]$, either there is $\sigma
_{0}\in(0,r)$ such that $L\left(  \sigma_{0}\right)  =\delta$, or
$\delta<L(\sigma)$ for all $\sigma\in\left(  0,r\right)  $. Take $r^{\prime
}=\sigma_{0}$ in the first case and $r^{\prime}$ any point in $(0,r)$ in the
second case. Let $B_{r^{\prime}}$ be the open geodesic ball of $M$,
$B_{r^{\prime}}\subset B_{r}$, with the same center of $B_{r}$. Consider the
correspondent $t_{0}>0$ and the function $w\in C^{2}\left(  \Lambda
_{0}\right)  \cap C^{0}\left(  \overline{\Lambda}_{0}\right)  $, $w|_{\partial
B_{r^{\prime}}}=0$, where $\Lambda_{0}=\left\{  q\in M-B_{r^{\prime}}%
;d_{M}\left(  q,\partial B_{r^{\prime}}\right)  \leq t_{0}\right\}  $ is such
that $\mathfrak{M}\left(  w\right)  \leq0$, as given in Corollary \ref{mir}.
If the second case occurs, the height of this $w$ is greater than $\delta$ at
its correspondent distance $t_{0}$ to $\partial B_{r^{\prime}}$ and, as $w$ is
radial with respect to the center of $B_{r^{\prime}}$, there is $t_{0}%
^{\prime}<t_{0}$ such that, at the distance $t_{0}^{\prime}$ of $\partial
B_{r^{\prime}}$, the height of $w$ is $\delta$. In any case, there is
$0<t_{0}^{\prime}\leq t_{0}$ such that, setting
\[
W_{r^{\prime}}\left(  q\right)  =\left\{
\begin{array}
[c]{c}%
w\left(  q\right)  \text{ if }q\in\Lambda_{0}^{\prime}\\
\delta\text{ if }q\in M-\Lambda_{0}^{\prime}%
\end{array}
\right.
\]
where
\[
\Lambda_{0}^{\prime}:=\left\{  q\in\Lambda;d\left(  q\right)  \leq
t_{0}^{\prime}\right\}  .
\]
satisfies the same properties of the function $W$ as given in Corollary
\ref{mir}. Note that it is possible to translate the graph of $W_{r^{\prime}}$
in the $\partial/\partial t$-direction ($e_{k}$-direction) in a way that its
height at infinity is in $[c,\delta)$ and such that $\Gamma$, the intersection
of the hypersurface resulting of this displacement with $\left\{
t\geq0\right\}  $ is such that
\[
\partial\Gamma=\Gamma\cap B_{r}=\partial B_{r^{\prime\prime}},
\]
with $B_{r^{\prime\prime}}\subset B_{r}$ a geodesic open ball of $M$ with the
same center of $B_{r}$ and radius $r^{\prime\prime}$, being $r^{\prime
}<r^{\prime\prime}<r$.\ Now, move $\Gamma$ keeping $\partial\Gamma$ on $M$ and
the center of $\partial\Gamma$ on the geodesic of $M$ linking the center of
$B_{r}$ to $q_{0}\in\partial U$, until $\partial\Gamma$ touches $\partial U$
at $q_{0}$ and call $\widetilde{\Gamma}$ this final hypersurface. Observe that
such displacement is an isometry in $M\times\mathbb{R}$. Denote by
$\widetilde{B}_{r^{\prime\prime}}$ the geodesic ball contained in $B_{r}$ of
radius $r^{\prime\prime}$ such that $\partial\widetilde{B}_{r^{\prime\prime}%
}=\partial\widetilde{\Gamma}$. We have then that $W_{r^{\prime\prime}%
}:M-\widetilde{B}_{r^{\prime\prime}}\longrightarrow\mathbb{R}$ is a
(generalized) supersolution relatively to $\mathfrak{M}$. Moreover, since our
translation in $\partial/\partial t$ direction is small enough, $W_{r^{\prime
\prime}}$ satisfies $\mathfrak{M}\left(  W_{r^{\prime\prime}}\right)  =0$ in
$\overline{\Lambda}_{0}^{\prime\prime}$, where $\Lambda_{0}^{\prime\prime}$ is
a neighborhood in $U$ such that $q_{0}\in\partial\Lambda_{0}^{\prime\prime}$.
In particular, $W_{r^{\prime\prime}}\in C^{\infty}\left(  \overline{\Lambda
}_{0}^{\prime\prime}\right)  $. From comparison principle, since $0\leq
w_{c}|_{\partial U}\leq W_{r^{\prime\prime}}|_{\partial U}$ and $c\leq
W_{r^{\prime\prime}}\left(  \infty\right)  $, we conclude that $0\leq
w_{c}\leq W_{r^{\prime\prime}}$ on $\overline{U}$. As $w_{c}\left(
q_{0}\right)  =W_{r^{\prime\prime}}\left(  q_{0}\right)  \ $and $W_{r^{\prime
\prime}}\in\nolinebreak C^{\infty}\left(  \overline{\Lambda}_{0}^{\prime
\prime}\right)  $ it follows that $w_{c}\in C^{1}\left(  \overline{U}\right)
$ and, from elliptic PDE regularity theory ( \cite{GT}), it follows that
$w_{c}\in C^{2,\alpha}\left(  \overline{U}\right)  \cap C^{\infty}\left(
U\right)  $.
\end{proof}

{\small Ari Jo\~{a}o Aiolfi }

{\small Departamento de Matem\'{a}tica }

{\small Universidade Federal de Santa Maria }

{\small Santa Maria RS/Brazil }

{\small ari.aiolfi@ufsm.br }

{\small ................................ }

{\small Caroline Maria Assmann }

{\small Instituto de Matem\'{a}tica }

{\small Universidade Federal do Rio Grande do Sul }

{\small Porto Alegre RS/Brazil }

{\small carol.assmann0504@gmail.com }

{\small ................................ }

{\small Jaime Bruck Ripoll }

{\small Instituto de Matem\'{a}tica }

{\small Universidade Federal do Rio Grande do Sul }

{\small Porto Alegre RS/Brazil }

{\small jaime.ripoll@ufrgs.br}

\begin{thebibliography}{99}                                                                                               %


\bibitem {ABR}A. Aiolfi, D. Bustos, J. Ripoll: On the existence of foliations
by solutions to the exterior Dirichlet problem for the minimal surface
equation, Proceedings of the AMS vol. 150, no.7, p. 3063-3073 (2022).

\bibitem {BGG}E. Bombieri, E. De Giorgi, E. Giusti: Minimal cones and the
Bernstein problem. Invent. Math. 7, p. 243--268 (1969).

\bibitem {CK}P. Collin, R. Krust: Le probl\`{e}me de Dirichlet pour
l'\'{e}quation des surfaces minimales sur des domaines non born\'{e}s Bull.
Soc. Math. France, 119, no. 4, p. 443-462 (1991).

\bibitem {DR}M. Dajczer, J. Ripoll: An extension of a theorem of Serrin to
graphs in warped products, J. Geom. Anal., 15 193--205 (2005).

\bibitem {DHL}M. Dajczer, P. Hinojosa, J.H de Lira: Killing graphs with
prescribed mean curvature, Calc. Var. Partial Differ. Eq. 33, p. 231--248 (2008).

\bibitem {DL}M. Dajczer, J.H de Lira: Killing graphs with prescribed mean
curvature and Riemannian submersions, Ann. Inst. H. Poincar\'{e} (Anal. Non
Lin\'{e}aire), vol. 26, no 3, p. 763-775 (2009).

\bibitem {ERo}R. Sa Earp, H, Rosenberg: The Dirichlet problem for th\'{e}
minimal surface equation on unbounded planar domains, J. Math. Pures Appl.,
68, p. 163-183 (1989).

\bibitem {HL}Wu-yi Hsiang, B. Lawson, Jr:\ Minimal submanifolds of low
cohomogeneity, Journal of Differential Geometry, p. 1-38 (1971).

\bibitem {CHHL}JB Casteras, E. Heinonen, I. Holopainen, J. Lira: Asymptotic
Dirichlet problems in warped products, Math. Z. 295, 211--248 (2020).

\bibitem {JS}H. Jenkins, J. Serrin: The Dirichlet problem for the minimal
surface equation in higher dimensions. J. Reine Angew. Math. 229, p. 170--187 (1968).

\bibitem {J}Z. Jin: Growth rate and existence of solutions to Dirichlet
problems for prescribed mean curvature equation on unbounded domains,
Electronic J. of Diff. Equations, 24, p. 1--15 (2008)

\bibitem {GT}D. Gilbarg, N. Trudinger: Elliptic Partial Differential Equations
of Second Order, Springer, Berlin, 1998.

\bibitem {K}E. Kuwert: On solutions of the exterior Dirichlet problem for the
minimal surface equation\textit{, }Ann. Inst. H. Poincar\'{e} (Anal. Non
Li\'{e}aire), vol. 10, no. 4, p. 445-451 (1993).

\bibitem {Kr}R. Krust: Remarques sur le probl\`{e}me ext\'{e}rieur de Plateau,
Duke Math. J., Vol. 59, pp. 161-173 (1989).

\bibitem {KT}N. Kutev and F. Tomi: Existence and nonexistence for the exterior
Dirichlet problem for the minimal surface equation in the plane, Differential
Integral Equations, 11, no. 6, p. 917--928 (1998).

\bibitem {N}J. C. C. Nitsche: Vorlesungen \"{u}ber Minimalfl\"{a}chen,
Grundlehren der mathematischen Wissenschaften 199, Springer-Verlag, Berlin (1975).

\bibitem {O}R. Osserman: A Survey of Minimal Surfaces, Van Nostrand Reinhold
Math. Studies 25, New York, 1969.

\bibitem {R}J. Ripoll: Some characterization, uniqueness and existence results
for euclidean graphs of constant mean curvature with planar boundary\emph{,
}Pacific Journal of Mathematics, Vol. 198, N. 1, 175-196 (2001).

\bibitem {RTa}J. Ripoll, F. Tomi: On solutions to the exterior Dirichlet
problem for the minimal surface equation with catenoidal ends, Adv. Calc. Var.
7, no. 2, p. 205--226 (2014).

\bibitem {RT}J. Ripoll, F. Tomi: Group invariant solutions of certain partial
differential equations, Pacific J. of Math 315 (1), p. 235-254 (2021).

\bibitem {S}Simons, J.: Minimal varieties in Riemannian manifolds. Ann. of
Math., 88, p. 62--105 (1968).
\end{thebibliography}
\end{document}